\documentclass{amsart}

\usepackage{mystyle}

\begin{document}
\begin{center}
\begin{LARGE}

{\bf A Proof of the KNS conjecture : $D_r$ case}

\end{LARGE}

\vspace{1cm}
\begin{large}
Chul-hee Lee
\end{large}

\vspace{0.6cm}
Max-Planck-Institut f\"ur Mathematik, 
Vivatsgasse 7, 53111 Bonn, Germany
\end{center}
%\title{A Proof of the KNS conjecture : $D_r$ case}
%\author{Chul-hee Lee}
%\address{Max-Planck-Institut f\"ur Mathematik, Vivatsgasse 7, 53111 Bonn, Germany}
%\email{chlee@mpim-bonn.mpg.de}
%\date{October 2012, last modified on \today}
%\maketitle

%PACS number: 02.20.Uw, 11.25.Hf
% 02.20.Uw Quantum groups 
% 11.25.Hf Conformal field theory, algebraic structures 
%Mathematics Subject Classification: 81R10, 81R50
%81R50, Quantum groups and related algebraic methods
%81R10, Infnite-dimensional groups and algebras motivated by physics, including Virasoro, Kac-Moody, W-algebras and other current algebras and their representations

\vspace{0.6cm}
\begin{center}
{\bf Abstract}
\end{center}
\vspace{0.4cm}
We prove the Kuniba-Nakanishi-Suzuki (KNS) conjecture concerning the quantum dimension solution of the $Q$-system of type $D_r$ obtained by a certain specialization of classical characters of the Kirillov-Reshetikhin modules. To this end, we use various symmetries of quantum dimensions. As a result, we obtain an explicit formula for the positive solution of the level $k$ restricted $Q$-system of type $D_r$ which plays an important role in dilogarithm identities for conformal field theories.

\section{introduction}
\subsection{Background and motivation}
Let $X$ be a simply laced Dynkin diagram with the index set $I=\{1,2,\cdots, r\}$. For a family of variables $\{ Q^{(a)}_{m}|a\in I, m\in \mathbb{Z}_{\geq 0} \}$, consider recurrences given by 
\begin{equation}\label{4:Qsys}
\left(Q^{(a)}_{m}\right)^2=\prod _{b\in I} \left(Q^{(b)}_{m}\right)^{\mathcal{I}(X)_{ab}}+Q^{(a)}_{m-1}Q^{(a)}_{m+1} \quad, m\ge 1
\end{equation}
where $\mathcal{I}(X)$ denotes the adjacency matrix of $X$. We call it the unrestricted $Q$-system of type $X$. Throughout the paper, we will use the boundary conditions $Q^{(a)}_{0} =1$ for all $a\in I$. 

Let $k\geq 1$ be an integer. We are interested in finding complex solutions of the $Q$-system satisfying another set of boundary conditions $Q^{(a)}_{k}=1$ for all $a\in I$. We define the level $k$ restricted $Q$-system of type $X$ to be the system of equations
\begin{equation}
\left\{
\begin{array}{lll}
Q^{(a)}_{0} =1 & a\in I \\
\left(Q^{(a)}_{m}\right)^2=\prod _{b\in I} \left(Q^{(b)}_{m}\right)^{\mathcal{I}(X)_{ab}}+Q^{(a)}_{m-1}Q^{(a)}_{m+1} & 1\le m <k, a\in I\\
Q^{(a)}_{k} =1 & a\in I
\end{array}
\right.
\end{equation}
in variables $\{Q^{(a)}_{m}| a\in I , 0\leq m \leq k\}$.

One reason to consider it comes from Nahm's conjecture about modularity of $q$-hypergeometric series. See \cite{Nahm, Keegan:2007zq, 2009arXiv0905.3776N} for instance. Another reason to study it can be found in dilogarithm identities for conformal field theories \cite{springerlink:10.1007/BF01840426, MR2804544}. To motivate our investigation, we give a brief exposition of dilogarithm identities related to our main results.

It is known that there exists a special unique solution of the level $k$ restricted $Q$-system possessing positivity and some additional properties as follows :
\begin{thm}\label{5:Quni}
Let $X$ be a Dynkin diagram of type $ADE$ of rank $r$. There exists a unique solution $\mathbf{z}=(z^{(a)}_{m})$ of the level $k$ restricted $Q$-system of type $X$ satisfying $z^{(a)}_{m}>0$ for $0\leq m \leq k$ and $a\in I$. 
For all $a\in I$, the following properties hold :
\begin{enumerate}
\item (symmetry) $z^{(a)}_{m}=z^{(a)}_{k-m}$ for $0\leq m \leq k$, 
\item (unimodality) 
$z^{(a)}_{m-1}<z^{(a)}_{m}$ for $1\le m \le \lfloor\frac{k}{2}\rfloor$ where $\lfloor x\rfloor$ is the floor function. 
\end{enumerate}
\end{thm}
See \cite[Theorem 5.3.6]{chlee2012} for a proof. We call the solution $\mathbf{z}=(z^{(a)}_{m})$ characterized in Theorem \ref{5:Quni} the positive solution of the level $k$ restricted $Q$-system.

The Rogers dilogarithm function is defined by 
$$L(x)=-\frac{1}{2}\int_{0}^{x}\frac{\log(1-y)}{y}+\frac{\log(y)}{1-y}\,dy$$
for $x\in (0,1)$. We set $L(0)=0$ and $L(1)=\pi^2/6$ so that $L$ is continuous on $[0,1]$. 

\begin{thm}\cite{MR2804544}
Let $X$ be a simply laced Dynkin diagram of rank $r$ and $\mathfrak{g}$ the corresponding simple Lie algebra. For $1\leq m \leq k-1$ and $a\in I$, let $$x_{m}^{(a)}=\frac{\prod_{b\in I} (z_{m}^{(b)})^{\mathcal{I}(X)_{ab}}}{(z_{m}^{(a)})^2}$$ 
where $\mathbf{z}=(z^{(a)}_{m})$ is the positive solution of the level $k$ restricted $Q$-system of type $X$.
The following dilogarithm identity holds :
\begin{equation}{\label{dilogcft}}
\frac{6}{\pi^2}\sum_{a \in I}\sum_{m=1}^{k-1}L(x_{m}^{(a)})=\frac{k \dim \mathfrak{g}}{h+k}-\operatorname{rank} \mathfrak{g}=\frac{(k-1)h r}{h+k}
\end{equation}
where $h$ denotes the Coxeter number of $\mathfrak{g}$.
\end{thm}
For a physical interpretation of the rational number on the right hand side of (\ref{dilogcft}), see \cite[Theorem 5.2]{1751-8121-44-10-103001} and references given there.

In this paper, we will study the positive solution of the level $k$ restricted $Q$-system of type $X$ using Lie theory. Before stating our main results, let us set up notation and terminology.

\subsection{Notation}
Let $X$ be a simply laced Dynkin diagram and $\mathcal{C}=(a_{ij})$ the Cartan matrix. Let $\mathfrak{g}$ be the corresponding simple Lie algebra of rank $r$ and $\mathfrak{h}$ its Cartan subalgebra. We denote the dual space of $\mathfrak{h}$ by $\mathfrak{h}^{*}$ and use the symbol $\langle \cdot ,\cdot \rangle$ to denote the natural pairing between $\mathfrak{h}$ and $\mathfrak{h}^{*}$. 

Let $\Phi\subset \mathfrak{h}^{*}$ be the root system. We denote the set of positive roots by $\Phi^{+}$ and the set of simple roots by $\Pi=\{\alpha_{i}|i\in I\}$. We will write $\alpha>0$ if $\alpha\in \Phi^{+}$ and $\alpha\geq \beta$ if $\alpha-\beta \in \Phi^{+}$ or $\alpha=\beta$.
For $\alpha=\sum_{i=1}^{r}c_i\alpha_{i}\in \Phi$, we define its height, denoted by $\operatorname{ht} \alpha$, to be $\sum_{i=1}^{r} c_i$.

Let $\theta=\sum_{i=1}^{r}a_i\alpha_{i}$ be the highest root. We call $a_i$ the marks and set $a_0=1$. 
See Figures \ref{pic:D} and \ref{pic:extD} for $X=D_r$. We denote the Coxeter number by $h$, which is given by $1+\operatorname{ht} \theta=\sum_{i=0}^{r}a_i$. Let $(\cdot |\cdot )$ be the standard symmetric bilinear form on $\mathfrak{h^{*}}$ normalized by requiring that $(\theta|\theta)=2$. 

Let $Q$ be the root lattice and $P$ be the weight lattice. The coroot lattice, denoted by $Q^{\vee}$ is the $\mathbb{Z}$-dual of the weight lattice $P$. We will choose the basis $\Pi^{\vee}=\{h_{i}\in \mathfrak{h}|i\in I\}$ of the coroot lattice so that $\langle \alpha_{i},h_j\rangle=a_{ji}$. Let $\{\omega_{i}\in P|i\in I\}$ be the dual basis of $P$ for $\Pi^{\vee}$ so that $\langle\omega_{i},h_{j}\rangle=\delta_{ij}$. We call $\omega_{i}$ the fundamental weights. A dominant weight is an element of $P_{+}=\{\sum_{i=1}^{r}\lambda_{i}\omega_{i}\in P|\lambda_{i}\ge 0, i\in I\}$. We call  $\rho=\sum_{i=1}^{r}\omega_{i}\in P$ the Weyl vector.
 
We have the group algebra $\mathbb{C}[P]$ with $\mathbb{C}$-basis of elements of the form $e^{\lambda}$, $\lambda \in P$. We can regard $e^{\lambda}$ as a function defined on $\mathfrak{h}^{*}$ by $\mu \mapsto e^{2\pi i (\lambda|\mu)}$. For a dominant weight $\lambda\in P_{+}$, the character $\chi_{\lambda}\in \mathbb{C}[P]$ of an irreducible representation $V$ of highest weight $\lambda$ is defined to be
$$\sum_{\lambda' \in \mathfrak{h}^{*}} (\dim{V_{\lambda'}})e^{\lambda'}$$
where $V_{\lambda'}$ denotes the weight space corresponding to $\lambda' \in \mathfrak{h}^{*}$. We will regard $\chi_{\lambda}$ as a function on $\mathfrak{h}^{*}$.

\begin{figure}\label{pic:D}
\begin{tikzpicture}
\begin{scope}[start chain]
\dnode{1}
\dnode{2}
\dnode{3}
\dydots
\dnode{r-2}
\dnode{r-1}
\end{scope}
\begin{scope}[start chain=br going above]
\chainin(chain-2);
\dnodebr{0};
\end{scope}
\begin{scope}[start chain=br2 going above]
\chainin(chain-5);
\dnodebr{r};
\end{scope}
\end{tikzpicture}
\caption{The extended Dynkin diagram of type $D_r^{(1)}$ and the labeling of the nodes}
\end{figure}
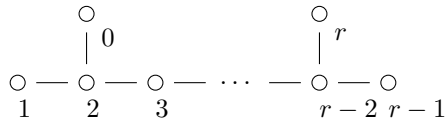

\subsection{Statement of the KNS conjecture and the main theorem}
Let $q$ be a non-zero complex number which is not a root of unity. The Kirillov-Reshetikhin (KR) modules form a special class of finite dimensional modules of the quantum affine algebra $U_{q}(\hat{\mathfrak{g}})$ and they are parametrized by $a\in I$, $m\in \mathbb{Z}_{\geq 0}$ and $u\in \mathbb{C}$. Since the quantized universal enveloping algebra $U_{q}(\mathfrak{g})$ is contained in $U_{q}(\hat{\mathfrak{g}})$ as a subalgebra, for a given KR module $W^{(a)}_{m}(u)$, we can get the finite dimensional $U_{q}(\mathfrak{g})$-module ${\rm res}\, W^{(a)}_{m}(u)$ by restriction.

The important point to note here is the fact that the classical characters $Q^{(a)}_m$ of ${\rm res}\, W^{(a)}_{m}(u)$ for $a\in I$ and $m\in \mathbb{Z}_{\geq 0}$ satisfy the unrestricted $Q$-system, which was first stated in \cite{Kirillov1990} and later proved in \cite{MR1993360} and \cite{MR2254805}.

The character $Q^{(a)}_m$ can be expanded into a sum of characters of irreducible modules of $\mathfrak{g}$ as
\begin{equation}\label{4:Qdecomposition2}
Q^{(a)}_{m}=\sum_{\omega \in P_{+}}Z(a, m,\omega)\chi_{\omega}
\end{equation}
where $Z(a, m,\omega)$ is a certain non-negative integer with $Z(a, m,m\omega_a)=1$.
For example, when $X=A_r$, we have $Q^{(a)}_{m}=\chi_{m \omega_{a}}$
for $a\in I$ and $m\in \mathbb{Z}_{\geq 0}$ and they satisfy the unrestricted $Q$-system of type $A_r$. 
For $X=D_r$, it is given by
\begin{equation}\label{6:qdecompD}
Q^{(a)}_{m}= \left\{
\begin{array}{lll}
\sum \chi_{k_{a}\omega_{a}+k_{a-2}\omega_{a-2}+\cdots + k_{1} \omega_{1}}&  1\leq a \leq r-2, a\equiv 1 \pmod 2, \\
\sum \chi_{k_{a}\omega_{a}+k_{a-2}\omega_{a-2}+\cdots + k_{0} \omega_{0}} & 1\leq a \leq r-2, a\equiv 0 \pmod 2, \\
\chi_{m \omega_{a}} & a=r-1,r
\end{array}
\right.
\end{equation}
where $\omega_0=0$ and the summation is over all nonnegative integers satisfying $k_{a}+k_{a-2}+\cdots + k_{1}=m$ for $a$ odd and $k_{a}+k_{a-2}+\cdots + k_{0}=m$ for $a$ even.

For a treatment of more general cases, see \cite[Appendix A]{MR1745263}, \cite[Section 13]{1751-8121-44-10-103001} and references given there.

If we regard $Q^{(a)}_{m}$ as a sum of characters given by (\ref{4:Qdecomposition2}), we can specialize it at the element $\frac{\rho}{h+k}\in \mathfrak{h}^{*}$. For each $a\in I$ and $m\in \mathbb{Z}_{\geq 0}$, we define $z^{(a)}_{m}$ by
$$z^{(a)}_{m}=Q^{(a)}_{m}(\frac{\rho}{h+k})=\sum_{\omega \in P_{+}}Z(a, m,\omega)\chi_{\omega}(\frac{\rho}{h+k}).$$ 
This yields a solution of the unrestricted $Q$-system and we call it the quantum dimension solution of the $Q$-system. In \cite{1751-8121-44-10-103001}, it has been conjectured that it gives the positive solution of the level $k$ restricted $Q$-system and satisfies some additional level truncation properties.
\begin{conjecture} \cite[Conjecture 14.2.]{1751-8121-44-10-103001}\label{KNSconj}
Let $z^{(a)}_{m}=Q^{(a)}_{m}(\frac{\rho}{h+k})$ for $a\in I$ and $m\in \mathbb{Z}_{\geq 0}$. For all $a\in I$, the following properties hold :
\begin{enumerate}
\item (positivity) $z^{(a)}_{m}>0$ for $0\leq m \leq k$,
\item (symmetry) $z^{(a)}_{m}=z_{k-m}^{(a)}$ for $1\leq m \leq k-1$,
\item (unit boundary condition) $z^{(a)}_{k}=1$,
\item (unimodality) 
$z_{m-1}^{(a)}<z^{(a)}_{m}$ holds true for $1\le m \le \lfloor\frac{k}{2}\rfloor$ where $\lfloor x\rfloor$ is the floor function,
\item (occurrence of 0)  $z^{(a)}_{k+1}=z^{(a)}_{k+2}=\cdots =z^{(a)}_{k+h-1}=0$.
\end{enumerate}
\end{conjecture}

We call Conjecture \ref{KNSconj} the KNS conjecture. The conjecture was originated from some unproven claims in \cite{springerlink:10.1007/BF01840426} whose motivation can be found in the study of thermodynamic properties of the RSOS models \cite{MR1017122}. Then it was subsequently formulated as above in \cite{MR1304818} with more general specializations. The conjecture had been proved only in the case of type $A_r$. In this paper, we will prove the following.

\begin{thmnn}\label{mainthmDr}
The KNS conjecture is true for $X=D_r$. Moreover, $z^{(a)}_{k+h}=1$ for $1\le a\le r-2$, 
$z^{(r-1)}_{k+h}=z^{(r)}_{k+h}=1$ when $r\equiv 0,1 \pmod 4$ and $z^{(r-1)}_{k+h}=z^{(r)}_{k+h}=-1$ when $r\equiv 2,3 \pmod 4$.
\end{thmnn}

Since our proof crucially depends on (\ref{4:Qdecomposition2}), we cannot properly deal with the exceptional types where the decompositions are still largely conjectural. There are more general $Q$-systems including non-simply laced types, for which we need to modify (\ref{4:Qsys}) into a slightly more complicated form. Thus we focus on the case of type $D_r$ to see the central idea clearly. A proof for other classical types will be given in a forthcoming paper.

This result will be divided into several parts and will be proved in Theorem \ref{main1}, \ref{main2} and \ref{main3}. The most tricky part lies in proving the unit boundary condition. Once we prove it, many results follow from Theorem \ref{5:Quni}. The key ingredients of our proof are the affine Weyl group symmetry and the extended Dynkin diagram symmetry of quantum dimensions of the affine weights obtained by suitable affinizations of classical weights. Although these are well-known concepts, they have not been effectively employed to attack our problem. 

In Section \ref{sec:qdim} we review necessary results about quantum dimensions. Section \ref{sec:comp} contains some preliminary calculations involving the affine Weyl group, which will be used in Section \ref{sec:mainpf} where we prove our main results.

\section{Review on quantum dimensions}\label{sec:qdim}
\begin{figure}\label{pic:extD}
\begin{tikzpicture}
\begin{scope}[start chain]
\dnode{1}
\dnode{2}
\dnode{2}
\dydots
\dnode{2}
\dnode{1}
\end{scope}
\begin{scope}[start chain=br going above]
\chainin(chain-2);
\dnodebr{1};
\end{scope}
\begin{scope}[start chain=br2 going above]
\chainin(chain-5);
\dnodebr{1};
\end{scope}
\end{tikzpicture}
\caption{The extended Dynkin diagram of type $D_r^{(1)}$ and the marks.}
\end{figure}
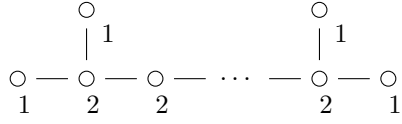

In this section we summarize some of the standard facts on quantum dimensions without proofs. For a thorough treatment we refer the reader to \cite[Section 16.3]{philippe1997conformal}.

\begin{definition}
Let $\hat{P}$ be the lattice generated by $\hat{\omega}_0,\hat{\omega}_1,\cdots, \hat{\omega}_r$. We define $\hat{P}^k$ to be $\{\sum_{i=0}^{r}\lambda_{i}\hat{\omega}_{i}\in \hat{P}|\sum_{i=0}^{r}a_{i}\lambda_{i}=k\}$ where $a_0=1$. 
We will denote the set $\{\sum_{i=0}^{r}\lambda_{i}\hat{\omega}_{i}\in \hat{P}^k|\lambda_{i}\ge 0\}$ by $\hat{P}_{+}^{k}$.
For a weight $\omega=\sum_{i=1}^{r}\lambda_{i}\omega_{i}\in P$, we define its level $k$ affinization $\hat{\omega}\in \hat{P}^k$ to be $\sum_{i=0}^{r}\lambda_{i}\hat{\omega}_{i}$. Note that $\lambda_0\in \mathbb{Z}$ is uniquely determined by the requirement $\hat{\omega} \in \hat{P}^k$.

Let $\alpha_0=-\theta$ and $\hat{\alpha}_j=\sum_{i=0}^{r} (\alpha_j|\alpha_i) \hat{\omega}_{i}$. We define the fundamental reflections $s_0,s_1,\cdots s_r$ on $\hat{P}$ linearly by 
$$s_i \hat{\omega}_j=\hat{\omega}_j -\delta_{ij}\hat{\alpha}_i$$
where $\delta_{ij}$ denotes the Kronecker delta. They generate the affine Weyl group $W$. The signature of $w\in W$ will be denoted by $(-1)^{\ell(w)}$ where $\ell(w)$ is the length of $w\in W$. 

Let $\hat{\rho}=\sum_{i=0}^{r}\hat{\omega}_{i}\in \hat{P}$. We define the shifted affine Weyl group action on the set $\hat{P}$ by $$w\cdot \hat{\lambda}=w(\hat{\lambda}+\hat{\rho})-\hat{\rho}$$ for $w\in W$. 
\end{definition}

\begin{definition}
Let $\lambda \in P$ and $\hat{\lambda}\in \hat{P}^{k}$ be its level $k$ affinization. The quantum dimension \index{quantum dimension} or $q$-dimension \index{$q$-dimension} of $\hat{\lambda}$ is defined by
\begin{equation}\label{6:qdimdef}
\mathcal{D}_{\hat{\lambda}}=\chi_{\lambda}\left(\frac{\rho}{h+k}\right)= \frac{\prod_{\alpha>0}\sin \frac{\pi(\lambda+\rho|\alpha)}{h+k}}{ \prod_{\alpha>0}\sin \frac{\pi (\rho|\alpha)}{h+k}}.
\end{equation}
\end{definition}

\begin{theorem} \label{6:qdimpos}
Let $\lambda=\sum_{i=1}^{l} \lambda_i \omega_i\in P_{+}$ be a dominant weight such that $\sum_{i=1}^{l}a_i\lambda_i \leq k$. For its level $k$ affinization $\hat{\lambda} \in \hat{P}_{+}^{k}$, $\mathcal{D}_{\hat{\lambda}}>0$. 
\end{theorem}

The shifted action of the affine Weyl group will be crucial in studying quantum dimensions as the following results show.

\begin{theorem}\label{6:Sfixed0}
For $\hat{\lambda}\in \hat{P}^{k}$ and $w\in W$, $\mathcal{D}_{w\cdot \hat{\lambda}}=(-1)^{\ell(w)}\mathcal{D}_{\hat{\lambda}}$. If $w\in W$ is an element of odd signature and $w\cdot \hat{\lambda}=\hat{\lambda}$, then $\mathcal{D}_{\hat{\lambda}}=0$.
\end{theorem}

\begin{theorem}\label{alcoverep}
If $\mathcal{D}_{\hat{\lambda}}\neq 0$ for $\hat{\lambda}\in \hat{P}^{k}$, then we can find a unique element $ \hat{\lambda}'\in \hat{P}_{+}^{k}$ such that $\hat{\lambda}'=w\cdot \hat{\lambda}$ for some $w\in W$.
\end{theorem}

We now look at the role of the symmetry of the extended Dynkin diagram.
\begin{theorem} \label{affinesym}
If $\hat{\lambda}_1$ and $\hat{\lambda}_2\in \hat{P}^{k}$ are conjugate by an automorphism of the extended Dynkin diagram, then $\mathcal{D}_{\hat{\lambda}_1}=\mathcal{D}_{\hat{\lambda}_2}$.
\end{theorem}

\begin{corollary} \label{kpluszero}
Let $\omega_i$ be a fundamental weight such that $\hat{\omega}_{i}$ is conjugate to $\hat{\omega}_{0}$ by an automorphism of the extended Dynkin diagram. If $\hat{\lambda}=k\hat{\omega}_i \in \hat{P}^{k}$, then $\mathcal{D}_{\hat{\lambda}}=1$. 
\end{corollary}

\section{Preliminary computations}\label{sec:comp}
From now on we will assume that $X=D_r$ and $$z^{(a)}_{m}=Q^{(a)}_{m}(\frac{\rho}{h+k})= \sum_{\omega\in \Omega^{(a)}_{m}} \mathcal{D}_{\hat{\omega}}$$
where $\Omega^{(a)}_{m}$ denotes the set of elements appearing in the sum (\ref{6:qdecompD}).

\begin{proposition}\label{tipsprop}
Let $a\in \{1,r-1,r\}$. The following properties hold :
\begin{align}
&z^{(a)}_{m}>0 \quad \text{for } 0\le m \le k, \label{1pos}\\
&z^{(a)}_m =z^{(a)}_{k-m} \quad \text{for }\, 0 < m < k ,\label{1sym}\\
&z^{(a)}_{k}=1. \label{1unit}
\end{align}
\end{proposition}

\begin{proof}
Note that $z^{(a)}_{m}=\mathcal{D}_{(k-m)\hat{\omega}_0+m\hat{\omega}_a}$.
The inequality (\ref{1pos}) follows from Theorem \ref{6:qdimpos}. By Theorem \ref{affinesym}, we have  $\mathcal{D}_{(k-m)\hat{\omega}_0+m\hat{\omega}_a}=\mathcal{D}_{m\hat{\omega}_0+(k-m)\hat{\omega}_a}$ and it implies (\ref{1sym}). (\ref{1unit}) is a consequence of Corollary \ref{kpluszero}.
\end{proof}

\begin{proposition}\label{1prop}
The following properties hold :
\begin{align}
&z^{(1)}_{k+j}=0 \qquad \text{for }\, 1 \le j \le h -1 \label{1zer0},\\
&z^{(1)}_{k+h}=1. \label{1reunit}
\end{align}
\end{proposition}

\begin{proof}
To prove (\ref{1zer0}), we use the product formula (\ref{6:qdimdef}) for the quantum dimension. First note that for each integer $l$ such that $1 \le l \le h-1=2r-1$, there exists a positive root $\alpha$ such that $\operatorname{ht} \alpha=l$ and $\alpha-\alpha_1 \ge 0$. Moreover, the number of such roots is exactly $h=2r-2$. In the product (\ref{6:qdimdef}) for $\mathcal{D}_{(k-m)\hat{\omega}_0+m\hat{\omega}_1}$, only those roots may contribute in a non-trivial way as
$$z^{(1)}_{m}=\mathcal{D}_{(k-m)\hat{\omega}_0+m\hat{\omega}_1}=\prod_{\substack{\alpha>0 \\ \alpha-\alpha_1\ge 0}}\frac{\sin\frac{\pi (\operatorname{ht}\alpha+m)}{h+k}}{\sin\frac{\pi \operatorname{ht}\alpha}{h+k}}.$$
Since $\{\operatorname{ht} \alpha+m| \alpha>0, \alpha-\alpha_1\ge 0\}$ is the same set as $\{n\in \mathbb{Z}| 1+m \leq n \leq (h-1)+m\}$, one can find a positive root $\alpha$ such that $\operatorname{ht} \alpha+m=h+k$ when $k+1 \le m \le k+(h-1)$. This proves (\ref{1zer0}).

We now turn to (\ref{1reunit}). If $m=h+k$, then
$$
\mathcal{D}_{(k-m)\hat{\omega}_0+m\hat{\omega}_1}=\prod_{\substack{\alpha>0 \\ \alpha-\alpha_1\ge 0}}\frac{\sin\frac{\pi(\operatorname{ht}\alpha+h+k)}{h+k}}{\sin\frac{\pi \operatorname{ht}\alpha}{h+k}}=\prod_{\substack{\alpha>0 \\ \alpha-\alpha_1\ge 0}}\frac{-\sin\frac{\pi \operatorname{ht}\alpha}{h+k}}{\sin\frac{\pi \operatorname{ht}\alpha}{h+k}}.
$$ Since this product is over $h=2r-2$ terms, the final product equals 1 and it proves $z^{(1)}_{k+h}=1$. 
\end{proof}

In the rest of the section, we will prove that for $2\le a \le r-2$, $z^{(a)}_{s}=z^{(a)}_{s+1}$ when $k$ is odd and $z^{(a)}_{s-1}=z^{(a)}_{s+1}$ when $k$ is even where $s=\lfloor\frac{k}{2}\rfloor$. If $2\le a \le r-2$, we will denote the element $k_a\hat{\omega}_{a}+k_{a-2}\hat{\omega}_{a-2}+\cdots + k_{2} \hat{\omega}_{2}+k_{0} \hat{\omega}_{0}\in \hat{P}^k$ by $(k_a,k_{a-2},\cdots, k_2, k_{0})$ when $a$ is even and $k_{a}\hat{\omega}_{a}+k_{a-2}\hat{\omega}_{a-2}+\cdots + k_{1} \hat{\omega}_{1} + k_{0} \hat{\omega}_{0}\in \hat{P}^k$ by $(k_a,k_{a-2},\cdots, k_1, k_{0})$ when $a$ is odd. Then we can write $$z^{(a)}_{m}=\sum_{\hat{\omega}\in \hat{\Omega}^{(a)}_{m}} \mathcal{D}_{\hat{\omega}}$$ where
\begin{equation}
\hat{\Omega}^{(a)}_{m}= \left\{(k_a,k_{a-2},\cdots, k_2, k_{0})\in \hat{P}^k\mid
\begin{array}{ll}
k_a+k_{a-2}+\cdots+k_2 \le m\\
k_a,k_{a-2},\cdots, k_2 \in \mathbb{Z}_{\geq 0} \\
\end{array}
\right\}
\end{equation}
for even $a$ such that $2\le a \le r-2$ and 
\begin{equation}
\hat{\Omega}^{(a)}_{m}= \left\{(k_a,k_{a-2},\cdots, k_1, k_{0})\in \hat{P}^k\mid 
\begin{array}{ll}
k_a+k_{a-2}+\cdots+k_1=m\\
k_a,k_{a-2},\cdots,k_1 \in \mathbb{Z}_{\geq 0} \\
\end{array}
\right\}
\end{equation}
for odd $a$ such that $2\le a \le r-2$.

Assume that the level $k$ is odd and $s=\frac{k-1}{2}$. 
\begin{prop}\label{evence}
If $a$ is even and $2\le a \le r-2$, then $z^{(a)}_{s}=z^{(a)}_{s+1}$.
\end{prop}

\begin{proof}
Note that $\hat{\Omega}^{(a)}_{s}\subseteq \hat{\Omega}^{(a)}_{s+1}$ and 
$$
\hat{\Omega}^{(a)}_{s+1} \setminus \hat{\Omega}^{(a)}_{s}= \left\{(k_a,k_{a-2},\cdots, k_2, k_{0})\in \hat{P}^k\mid
\begin{array}{ll}
k_a+k_{a-2}+\cdots+k_2 = s+1 \\
k_a,k_{a-2},\cdots, k_2 \in \mathbb{Z}_{\geq 0} \\
\end{array}
\right\}.
$$
If $\hat{\omega}=(k_a,k_{a-2},\cdots, k_2, k_{0})\in \hat{\Omega}^{(a)}_{s+1} \setminus \hat{\Omega}^{(a)}_{s}$, then $k_0=-1$ because $$k_0+2(k_a+k_{a-2}\cdots+k_2)=k=2s+1.$$ 
So for any $\hat{\omega}\in \hat{\Omega}^{(a)}_{s+1} \setminus \hat{\Omega}^{(a)}_{s}$, $\mathcal{D}_{\hat\omega}=0$ since $s_0\cdot \hat\omega=\hat\omega$.
 Thus $z^{(a)}_{s+1}=z^{(a)}_{s}$.
\end{proof}

\begin{prop}\label{oddce}
If  $a$ is odd and $2\le a \le r-2$, then $z^{(a)}_{s}=z^{(a)}_{s+1}$.
\end{prop}

\begin{proof}
Let $\hat{\omega}=(k_a,k_{a-2},\cdots, k_1, k_{0})\in \hat{\Omega}^{(a)}_{s+1}$. The condition 
\begin{equation}\label{2s1}
k_0+k_1+2(k_a+k_{a-2}+\cdots+k_3)=k_0+k_1+2(s+1-k_1)=k=2s+1
%published version has a typo k_{a_2}
%k_0+k_1+2(k_a+k_{a_2}+\cdots+k_3)=k_0+k_1+2(s+1-k_1)=k=2s+1
\end{equation}
implies $k_0=k_1-1$. If $k_1=0$, then $k_0=-1$ and $\mathcal{D}_{\hat\omega}=0$ since $s_0\cdot \hat\omega=\hat\omega$.

Let $$
(\hat{\Omega}^{(a)}_{s+1})'= \left\{(k_a,k_{a-2},\cdots, k_1, k_{0})\in \hat{P}^k\mid
\begin{array}{ll}
k_a+k_{a-2}+\cdots+k_1=s+1 \\
k_a,k_{a-2},\cdots,k_3 \in \mathbb{Z}_{\geq 0},k_1\ge 1 \\
\end{array}
\right\}
$$
and 

$$
\hat{\Omega}^{(a)}_{s}= \left\{(k_a,k_{a-2},\cdots, k_1, k_{0})\in \hat{P}^k\mid 
\begin{array}{ll}
k_a+k_{a-2}+\cdots+k_1=s\\
k_a,k_{a-2},\cdots,k_1 \in \mathbb{Z}_{\geq 0} \\
\end{array}
\right\}.
$$

Let us construct a bijection between $(\hat{\Omega}^{(a)}_{s+1})'$ and $\hat{\Omega}^{(a)}_{s}$. Define a map from $(\hat{\Omega}^{(a)}_{s+1})'$ to $\hat{\Omega}^{(a)}_{s}$ by
\begin{equation}\label{map01}
(k_a,k_{a-2},\cdots, k_1, k_{0}) \mapsto (k_a,k_{a-2},\cdots, k_0, k_1).
\end{equation}
To see that the map is well-defined, note that if $(k_a,k_{a-2},\cdots, k_1, k_{0})\in (\hat{\Omega}^{(a)}_{s+1})'\subset \hat{\Omega}^{(a)}_{s+1}$, then (\ref{2s1}) implies $k_0=k_1-1\ge 0$. Since 
$$
k_a+k_{a-2}+\cdots+k_3+ k_{0}=k_a+k_{a-2}+\cdots+k_3+ (k_{1}-1)=s,
$$
we have $(k_a,k_{a-2},\cdots, k_0, k_1)\in \hat{\Omega}^{(a)}_{s}$. The map (\ref{map01}) is injective since $k_0=k_1-1$. 

Conversely, any element $(k_a,k_{a-2},\cdots, k_1, k_{0})\in \hat{\Omega}^{(a)}_{s}$ satisfies $$k_0+k_1+2(s-k_1)=k_0-k_1+2s=k=2s+1.$$ Thus $k_0=k_1+1\ge 1$ which shows that (\ref{map01}) is surjective. We thus have proved that (\ref{map01}) is a bijection between $(\hat{\Omega}^{(a)}_{s+1})'$ and $\hat{\Omega}^{(a)}_{s}$.

By Theorem \ref{affinesym}, we have $\mathcal{D}_{(k_a,k_{a-2},\cdots, k_1, k_{0})}=\mathcal{D}_{(k_a,k_{a-2},\cdots, k_0, k_1)}$. This proves our assertion.
\end{proof}

We now assume that the level $k$ is even and $s=\frac{k}{2}$. 
\begin{lemma}\label{koddm4}
Let $a$ be even and $2\le a \le r-2$. If $\hat\omega=(k_a,k_{a-2},\cdots, k_2, -2)\in \hat{P}^k$ satisfies $k_a+k_{a-2}+\cdots+k_2=s+1$ and $k_2=0$, then $\mathcal{D}_{\hat\omega}=0$.
\end{lemma}
\begin{proof}
It is easy to check that $(s_0s_2s_0)\cdot \hat\omega=\hat\omega$. Theorem \ref{6:Sfixed0} now  gives the desired conclusion.
\end{proof}

\begin{prop}\label{evence2}
If $a$ is even and $2\le a \le r-2$, then $z^{(a)}_{s-1}=z^{(a)}_{s+1}$.
\end{prop}

\begin{proof}
Recall that
$$
\hat{\Omega}^{(a)}_{s-1}= \left\{(k_a,k_{a-2},\cdots, k_2, k_{0})\in \hat{P}^k\mid
\begin{array}{ll}
k_a+k_{a-2}+\cdots+k_2 \le s-1\\
k_a,k_{a-2},\cdots, k_2 \in \mathbb{Z}_{\geq 0} \\
\end{array}
\right\}
$$
and 
$$
\hat{\Omega}^{(a)}_{s+1}= \left\{(k_a,k_{a-2},\cdots, k_2, k_{0})\in \hat{P}^k\mid
\begin{array}{ll}
k_a+k_{a-2}+\cdots+k_2 \le s+1\\
k_a,k_{a-2},\cdots, k_2 \in \mathbb{Z}_{\geq 0} \\
\end{array}
\right\}.
$$

Let us define three disjoint subsets $R,S$ and $T$ of $\hat{\Omega}^{(a)}_{s+1}$ by
\begin{align}
&R=\{(k_a,k_{a-2},\cdots, k_2, k_0)\in \hat{\Omega}^{(a)}_{s+1}|k_a+k_{a-2}+\cdots+k_2=s+1, k_2 = 0\}, \notag\\
&S=\{(k_a,k_{a-2},\cdots, k_2, k_0)\in \hat{\Omega}^{(a)}_{s+1}|k_a+k_{a-2}+\cdots+k_2=s+1, k_2 \ge 1\}, \notag\\
&T=\{(k_a,k_{a-2},\cdots, k_2, k_0)\in \hat{\Omega}^{(a)}_{s+1}|k_a+k_{a-2}+\cdots+k_2=s\}. \notag
\end{align}

For $\hat{\omega}\in R$, $\mathcal{D}_{\hat\omega}=0$ by Lemma \ref{koddm4} and so $\sum_{\hat{\omega}\in R} \mathcal{D}_{\hat{\omega}}=0$. We now want to prove $\sum_{\hat{\omega}\in S\cup T} \mathcal{D}_{\hat{\omega}}=0$. For $(k_a,k_{a-2},\cdots, k_2, k_{0})\in \hat{P}^k$, we have $k_0+2(k_a+k_{a-2}\cdots+k_2)=k=2s$. If $(k_a,k_{a-2},\cdots, k_2, k_0)\in S$, then $k_0=-2$. For $(k_a,k_{a-2},\cdots, k_2,k_0)\in T$, we have $k_0=0$. We have a bijection between $S$ and $T$ since $$s_0\cdot (k_a,k_{a-2},\cdots, k_2, -2)=(k_a,k_{a-2},\cdots, k_2-1, 0).$$
By Theorem \ref{6:Sfixed0}, $\sum_{\hat{\omega}\in S\cup T} \mathcal{D}_{\hat{\omega}}=0$. Consequently, $$z^{(a)}_{s+1}=\sum_{\hat{\omega}\in \hat{\Omega}^{(a)}_{s+1}} \mathcal{D}_{\hat{\omega}}=\sum_{\hat{\omega}\in \hat{\Omega}^{(a)}_{s+1}\setminus (R\cup S\cup T)} \mathcal{D}_{\hat{\omega}}.$$

From $\hat{\Omega}^{(a)}_{s-1}\subset \hat{\Omega}^{(a)}_{s+1}$ and $\hat{\Omega}^{(a)}_{s-1}=\hat{\Omega}^{(a)}_{s+1}\setminus (R\cup S\cup T)$, we obtain $z^{(a)}_{s+1}=z^{(a)}_{s-1}$.
\end{proof}

\begin{lemma}\label{koddm5}
Let $a$ be odd and $2\le a \le r-2$. If $\hat\omega=(k_a,k_{a-2},\cdots,1, -1)\in \hat{P}^k$ or $\hat\omega=(k_a,k_{a-2},\cdots, 0, -2)\in \hat{P}^k$, then $\mathcal{D}_{\hat\omega}=0$.
\end{lemma}
\begin{proof}
For $\hat\omega=(k_a,k_{a-2},\cdots,1, -1)\in \hat{P}^k$, it is easy to see that $s_0\cdot \hat\omega=\hat\omega$. For $\hat\omega=(k_a,k_{a-2},\cdots, 0, -2)\in \hat{P}^k$, we can show $(s_0s_2s_0)\cdot \hat\omega=\hat\omega$. The lemma follows from Theorem \ref{6:Sfixed0}.
\end{proof}

\begin{prop}\label{oddce2}
If  $a$ is odd and $2\le a \le r-2$, then $z^{(a)}_{s-1}=z^{(a)}_{s+1}$
\end{prop}

\begin{proof}
For any $\hat{\omega}=(k_a,k_{a-2},\cdots, k_1, k_0)\in \hat{\Omega}^{(a)}_{s+1}$ with $k_1=0$ or $k_1=1$, $\mathcal{D}_{\hat\omega}=0$ by Lemma \ref{koddm5}.

Let us define $(\hat{\Omega}^{(a)}_{s+1})'$ by 
$$\left\{(k_a,k_{a-2},\cdots, k_1, k_{0})\in \hat{P}^k\mid
\begin{array}{ll}
k_a+k_{a-2}+\cdots+k_1=s+1 \\
k_a,k_{a-2},\cdots,k_3 \in \mathbb{Z}_{\geq 0},k_1\ge 2 \\
\end{array}
\right\}.
$$

Then we can write $z^{(a)}_{s+1}=\sum_{\hat{\omega}\in (\hat{\Omega}^{(a)}_{s+1})'} \mathcal{D}_{\hat{\omega}}$.

Let us define a map from $(\hat{\Omega}^{(a)}_{s+1})'$ to $\hat{\Omega}^{(a)}_{s-1}$ by 
\begin{equation}\label{s111}
(k_a,k_{a-2},\cdots, k_1, k_{0}) \mapsto (k_a,k_{a-2},\cdots, k_0, k_1).
\end{equation}
For $(k_a,k_{a-2},\cdots, k_1, k_{0})\in (\hat{\Omega}^{(a)}_{s+1})'$,
$$k_0+k_1+2(k_a+k_{a-2}+\cdots+k_3)=k_0+k_1+2(s+1-k_1)=k=2s.$$
Hence $k_0=k_1-2\ge 0$ and it shows that $(k_a,k_{a-2},\cdots, k_0, k_1)\in \hat{\Omega}^{(a)}_{s-1}$ and thus the map (\ref{s111}) is well-defined. It is clear that this is injective.

Conversely, any element $(k_a,k_{a-2},\cdots, k_1, k_{0})\in \hat{\Omega}^{(a)}_{s-1}$ satisfies $$k_0+k_1+2(s-1-k_1)=k_0-k_1+2s-2=2s.$$ Thus $k_0=k_1+2\ge 2$ and it proves that (\ref{s111}) is surjective and thus bijective.

By Theorem \ref{affinesym}, $\mathcal{D}_{(k_a,k_{a-2},\cdots, k_1, k_0)}=\mathcal{D}_{(k_a,k_{a-2},\cdots, k_0, k_1)}$ and it proves our proposition.
\end{proof}

\section{proof of the main theorem}\label{sec:mainpf}
In this section, we prove our main theorem using the results obtained in the previous section.

\begin{lemma}\label{zeq}
Let $\mathbf{w}=(w^{(a)}_{m})$ be a solution of the level $k$ restricted $Q$-system such that $w^{(a)}_{m}\neq 0$ for $0\leq m\leq k$ and $a\in I$. If $w^{(a)}_{1}=Q^{(a)}_{1}$ for any $a\in I$ and $\{ Q^{(a)}_{m}|a\in I, m\in \mathbb{Z}_{\geq 0} \}$ satisfies the unrestricted $Q$-system, then $w^{(a)}_{m}=Q^{(a)}_{m}$ for $0\leq m\leq k$ and $a\in I$. In particular, $Q^{(a)}_{k}=1$.
\end{lemma}

\begin{proof}
This is a direct consequence of the recursion (\ref{4:Qsys})
$$Q^{(a)}_{m+1}=\frac{(Q^{(a)}_{m})^2-\prod _{b\in I} (Q^{(b)}_{i})^{\mathcal{I}(X)_{ab}}}{Q^{(a)}_{m-1}}.$$
\end{proof}

\begin{theorem}\label{main1}
For all $a\in I$, the following properties hold :
\begin{enumerate}
\item (positivity) $z^{(a)}_{m}>0$ for $0\leq m \leq k$,
\item (symmetry) $z^{(a)}_{m}=z_{k-m}^{(a)}$ for $1\leq m \leq k-1$,
\item (unit boundary condition) $z^{(a)}_{k}=1$,
\item (unimodality) 
$z_{m-1}^{(a)}<z^{(a)}_{m}$ for $1\le m \le \lfloor\frac{k}{2}\rfloor$.
\end{enumerate}
\end{theorem}

\begin{proof}
To prove the unit boundary condition, we divide the argument into two cases when $k$ is odd and $k$ is even.

Assume first that $k$ is odd and $s=\frac{k-1}{2}$. For $0\leq m\leq k$ and $a\in I$, let us define $w^{(a)}_{m}$ by
$$
w^{(a)}_{m}=\left\{
\begin{array}{ll}
z^{(a)}_{m} & 0\leq m\leq s\\
z^{(a)}_{k-m}& s<m\leq k\\
\end{array}
\right..
$$

Since we have $z^{(a)}_{s}=z^{(a)}_{s+1}$ for any $a\in I$ by Proposition \ref{tipsprop}, \ref{evence} and \ref{oddce}, $\mathbf{w}=(w^{(a)}_{m})$ must be a solution of the level $k$ restricted $Q$-system. Since $w^{(a)}_{1}=z^{(a)}_{1}$ and $w^{(a)}_{m}>0$ for all $0\leq m\leq k$, we can conclude that $z^{(a)}_{m}=w^{(a)}_{m}$ for $0\leq m\leq k$ by Lemma \ref{zeq}. Especially, $z^{(a)}_{k}=1$ for any $a\in I$.
 
Assume that $k$ is even and let $s=\frac{k}{2}$. For $0\leq m\leq k$ and $a\in I$, let
\begin{equation}
w^{(a)}_{m}=\left\{
\begin{array}{ll}
z^{(a)}_{m} & 0\leq m\leq s\\
z^{(a)}_{k-m}& s<m\leq k\\
\end{array}
\right..
\end{equation}
 
Since we have $z^{(a)}_{s-1}=z^{(a)}_{s+1}$ for $a\in I$ by Proposition \ref{tipsprop}, \ref{evence2} and \ref{oddce2}, $\mathbf{w}=(w^{(a)}_{m})$ is a solution of the level $k$ restricted $Q$-system. By the same argument as above, we can conclude that the unit boundary condition $z^{(a)}_{k}=1$ holds for any $a\in I$. 

The properties of positivity and symmetry can be easily obtained from the definition for $w^{(a)}_{m}$. Now we know that $\mathbf{w}=\mathbf{z}=(z^{(a)}_{m})$ is the positive solution of the level $k$ restricted $Q$-system characterized in Theorem \ref{5:Quni} and it follows that $z_{m-1}^{(a)}<z^{(a)}_{m}$ for $1\le m \le \lfloor\frac{k}{2}\rfloor$.
\end{proof}

Now we prove $z^{(a)}_{k+j}=0$ for any $a\in I$ and $1 \le j \le h -1$.
\begin{proposition}\label{kpl10}
For all $a\in I$, $z^{(a)}_{k+1}=0$.
\end{proposition}
\begin{proof}
Since $z^{(a)}_{k}=1$ and $z^{(a)}_{k-1}\neq 0$ for all $a\in I$ by Theorem \ref{main1}, the recursion (\ref{4:Qsys})
$$(z^{(a)}_{k})^2=\prod _{b\in I} (z^{(b)}_{k})^{\mathcal{I}(X)_{ab}}+z^{(a)}_{k-1}z^{(a)}_{k+1}$$
implies $z^{(a)}_{k+1}=0$.
\end{proof}

\begin{lemma}\label{arr0}
Let $\{ Q^{(a)}_{m}|a\in I, m\in \mathbb{Z}_{\geq 0} \}$ be a solution of the unrestricted $Q$-system. The following condition
$$\left[\begin{array}{ccc}
Q^{(a-1)}_{m-1} & Q^{(a)}_{m-1}  & Q^{(a+1)}_{m-1} \\
Q^{(a-1)}_{m}  & Q^{(a)}_{m} & Q^{(a+1)}_{m} \\
Q^{(a-1)}_{m+1}  & Q^{(a)}_{m+1} & Q^{(a+1)}_{m+1}
\end{array}\right]=\left[\begin{array}{ccc}
 * & 0 & * \\
 0 & Q^{(a)}_{m} & * \\
 * & * & *
\end{array}\right]$$
implies $Q^{(a)}_{m}=0$ for $2\leq a\leq r-2$ where $*$ denotes an arbitrary number.
Similarly, the condition
$$\left[\begin{array}{ccc}
Q^{(r-2)}_{i-1} & Q^{(r-1)}_{i-1}  & Q^{(r)}_{i-1} \\
Q^{(r-2)}_{i}  & Q^{(r-1)}_{i} & Q^{(r)}_{i} \\
Q^{(r-2)}_{i+1}  & Q^{(r-1)}_{i+1} & Q^{(r)}_{i+1}
\end{array}\right]=\left[\begin{array}{ccc}
 * & 0 & 0 \\
 0 & Q^{(r-1)}_{m} & Q^{(r)}_{m} \\
 * & * & *
\end{array}\right]$$ 
implies $Q^{(r-1)}_{m}=Q^{(r)}_{m}=0$.
\end{lemma}

\begin{proof}
This is again a consequence of the recursion (\ref{4:Qsys}) for $Q^{(a)}_{m}$:
$$(Q^{(a)}_{m})^2=\prod _{b\in I} (Q^{(b)}_{i})^{\mathcal{I}(X)_{ab}}+Q^{(a)}_{m-1}Q^{(a)}_{m+1}.$$
In both cases, we obtain $(Q^{(a)}_{m})^2=0$.
\end{proof}

\begin{theorem}\label{main2}
$z^{(a)}_{k+j}=0$ for any $a\in I$ and $1 \le j \le h -1$.
\end{theorem}

\begin{proof}
Recall (\ref{1zer0}) that $z^{(1)}_{k+j}=0$ for $1 \le j \le h -1$. Since $z^{(a)}_{k+1}=0$ for any $a\in I$ by Proposition \ref{kpl10}, we get $z^{(a)}_{k+2}=0$ by applying Lemma \ref{arr0}. Repeated application of Lemma \ref{arr0} enables us to prove $z^{(a)}_{k+j}=0$ for any $a\in I$ and $1 \le j \le h -1$.
\end{proof}

For the following lemma, let us set $Q^{(0)}_{m}=1$ for convenience.
\begin{lemma}\label{arr1}
Let $\{ Q^{(a)}_{m}|a\in I, m\in \mathbb{Z}_{\geq 0} \}$ be a solution of the unrestricted $Q$-system. The following condition
$$
\left[\begin{array}{cccc}
Q^{(a-2)}_{m-1} & Q^{(a-1)}_{m-1} & Q^{(a)}_{m-1}  \\
Q^{(a-2)}_{m} & Q^{(a-1)}_{m}  & Q^{(a)}_{m}  \\
Q^{(a-2)}_{m+1} & Q^{(a-1)}_{m+1}  & Q^{(a)}_{m+1} 
\end{array}\right]=
\left[\begin{array}{cccc}
 * & 0 & 0 \\
 1 & 1 & Q^{(a)}_{m}  \\
 * & * & *
\end{array}\right]$$
implies $Q^{(a)}_{m}=1$ for $2\leq a\leq r-2$ where $*$ denotes an arbitrary number.
\end{lemma}

\begin{proof}
Let us look at (\ref{4:Qsys}) for $Q^{(a-1)}_{m}$,
$$(Q^{(a-1)}_{m})^2=Q^{(a-2)}_{m}Q^{(a)}_{m}+Q^{(a-1)}_{m-1}Q^{(a-1)}_{m+1}.$$
Under the condition stated above, we obtain $1^2=1\cdot Q^{(a)}_{m}+0$. This proves $Q^{(a)}_{m}=1$.
\end{proof}

\begin{theorem}\label{main3}
$z^{(a)}_{k+h}=1$ for $1\le a\le r-2$ and $z^{(r-1)}_{k+h}=z^{(r)}_{k+h}=\pm 1$.
\end{theorem}
\begin{proof}
Since $z^{(1)}_{k+h}=1$ by (\ref{1reunit}) and $z^{(a)}_{k+h-1}=0$ for any $a\in I$ by Theorem \ref{main2}, we get $z^{(a)}_{k+h}=1$ for $1\le a \le r-2$ by applying Lemma \ref{arr1}. The recursions  (\ref{4:Qsys}) for
 $z^{(r-2)}_{k+h}$, $z^{(r-1)}_{k+h}$ and $z^{(r)}_{k+h}$ give the following system of equations
 
$$
\left\{
\begin{array}{lll}
&(z^{(r-2)}_{k+h})^2=z^{(r-1)}_{k+h}z^{(r)}_{k+h}\\
&(z^{(r-1)}_{k+h})^2=z^{(r-2)}_{k+h}\\
&(z^{(r)}_{k+h})^2=z^{(r-2)}_{k+h}
\end{array}
\right..
$$
Then $z^{(r-2)}_{k+h}=1$ implies $z^{(r-1)}_{k+h}=z^{(r)}_{k+h}=\pm 1$.
\end{proof}

\begin{remark}
Analysis similar to that in the proof of Proposition \ref{1prop} using the product formula (\ref{6:qdimdef}) for the quantum dimension can be used to show $z^{(r-1)}_{k+h}=z^{(r)}_{k+h}=1$ when $r\equiv 0,1 \pmod 4$ and $z^{(r-1)}_{k+h}=z^{(r)}_{k+h}=-1$ when $r\equiv 2,3 \pmod 4$.
\end{remark}

\begin{example}
Let $X=D_{5}$ and $k=4$. We express $z_m^{(a)}$ in terms of quantum dimensions for $a\in I$ and $0\le m \le h+k=12$. For $a=1,4,5$, we get
$$
\begin{bmatrix}
 z_0^{(1)} & z_0^{(4)} & z_0^{(5)} \\
 z_1^{(1)} & z_1^{(4)} & z_1^{(5)} \\
 z_2^{(1)} & z_2^{(4)} & z_2^{(5)} \\
 z_3^{(1)} & z_3^{(4)} & z_3^{(5)} \\
 z_4^{(1)} & z_4^{(4)} & z_4^{(5)} \\
 z_5^{(1)} & z_5^{(4)} & z_5^{(5)} \\
\vdots & \vdots & \vdots \\
 z_{11}^{(1)} & z_{11}^{(4)} & z_{11}^{(5)} \\
 z_{12}^{(1)} & z_{12}^{(4)} & z_{12}^{(5)}
\end{bmatrix}
=
\begin{bmatrix}
\mathcal{D}_{4 \hat{\omega }_0} & \mathcal{D}_{4 \hat{\omega }_0} & \mathcal{D}_{4 \hat{\omega }_0} \\
 \mathcal{D}_{3 \hat{\omega }_0+\hat{\omega }_1} & \mathcal{D}_{3 \hat{\omega }_0+\hat{\omega }_4} & \mathcal{D}_{3 \hat{\omega }_0+\hat{\omega }_5} \\
 \mathcal{D}_{2 \hat{\omega }_0+2 \hat{\omega }_1} & \mathcal{D}_{2 \hat{\omega }_0+2 \hat{\omega }_4} & \mathcal{D}_{2 \hat{\omega }_0+2 \hat{\omega }_5} \\
 \mathcal{D}_{\hat{\omega }_0+3 \hat{\omega }_1} & \mathcal{D}_{\hat{\omega }_0+3 \hat{\omega }_4} & \mathcal{D}_{\hat{\omega }_0+3 \hat{\omega }_5} \\
 \mathcal{D}_{4 \hat{\omega }_1} & \mathcal{D}_{4 \hat{\omega }_4} & \mathcal{D}_{4 \hat{\omega }_5} \\
 0 & 0 & 0 \\
 \vdots &  \vdots &  \vdots \\
 0 & 0 & 0 \\
 \mathcal{D}_{4 \hat{\omega }_1} & \mathcal{D}_{4 \hat{\omega }_4} & \mathcal{D}_{4 \hat{\omega }_5}
\end{bmatrix}.
$$
For $a=2,3$, we have
$$
\begin{bmatrix}
 z_0^{(2)} & z_0^{(3)} \\
 z_1^{(2)} & z_1^{(3)} \\
 z_2^{(2)} & z_2^{(3)} \\
 z_3^{(2)} & z_3^{(3)} \\
 z_4^{(2)} & z_4^{(3)} \\
 z_5^{(2)} & z_5^{(3)} \\
\vdots & \vdots \\
 z_{11}^{(2)} & z_{11}^{(3)} \\
 z_{12}^{(2)} & z_{12}^{(3)}
\end{bmatrix}=
\begin{bmatrix}
 \mathcal{D}_{4 \hat{\omega }_0} & \mathcal{D}_{4 \hat{\omega }_0}  \\
 \mathcal{D}_{4 \hat{\omega }_0}+\mathcal{D}_{2 \hat{\omega }_0+\hat{\omega }_2} & \mathcal{D}_{3 \hat{\omega }_0+\hat{\omega }_1}+\mathcal{D}_{2 \hat{\omega }_0+\hat{\omega }_3} \\
 \mathcal{D}_{4 \hat{\omega }_0}+\mathcal{D}_{2 \hat{\omega }_2}+\mathcal{D}_{2 \hat{\omega }_0+\hat{\omega }_2} & \mathcal{D}_{2 \hat{\omega }_0+2 \hat{\omega }_1}+\mathcal{D}_{2 \hat{\omega }_3}+\mathcal{D}_{\hat{\omega }_0+\hat{\omega }_1+\hat{\omega }_3} \\
 \mathcal{D}_{4 \hat{\omega }_0}+\mathcal{D}_{2 \hat{\omega }_0+\hat{\omega }_2} & \mathcal{D}_{\hat{\omega }_0+3 \hat{\omega }_1}+\mathcal{D}_{2 \hat{\omega }_1+\hat{\omega }_3} \\
 \mathcal{D}_{4 \hat{\omega }_0} & \mathcal{D}_{4 \hat{\omega }_1} \\
 0 & 0 \\
 \vdots &  \vdots\\
 0 & 0 \\
 \mathcal{D}_{4 \hat{\omega }_0} & \mathcal{D}_{4 \hat{\omega }_1}
\end{bmatrix}.$$
This expression is obtained by applying the shifted affine Weyl group action together with Theorem \ref{6:Sfixed0} and \ref{alcoverep}.
\end{example}

\section*{Acknowledgements}
This work is an outgrowth of the author's doctoral research and he would like to thank his advisor Richard E. Borcherds. He is also grateful to Tomoki Nakanishi for helpful discussions at the MSRI workshop on Cluster Algebras and Commutative Algebra. The work is partially supported by Samsung Scholarship.
\bibliographystyle{amsalpha}
\bibliography{KNSconjecture}
\end{document}